\documentclass[11pt]{amsart}
\usepackage[cm]{fullpage}
\usepackage{amssymb,amsmath,amsthm,amscd,mathrsfs,graphicx}
\usepackage[cmtip,all]{xy}
\usepackage{pb-diagram}
\usepackage{tikz}

\numberwithin{equation}{section}
\newtheorem{teo}{Theorem}[section]

\newtheorem{theorem}{Theorem}[section]
\newtheorem{proposition}{Proposition}[section]
\newtheorem{lemma}{Lemma}[section]

\newtheorem{question}{Question}[section]

\newcommand{\al}{\alpha}
\newcommand{\bal}{\boldsymbol{\alpha}}
\newcommand{\bbe}{\boldsymbol{\beta}}

\newcommand{\ov}[1]{\overline{#1}}

\newcommand{\ve}{\varepsilon}

\theoremstyle{definition}

\theoremstyle{remark}
\newtheorem{remark}[teo]{Remark}

\begin{document}
\bibliographystyle{amsplain}\

\title{Smooth solutions of  degenerate linear parabolic \\ equations and the porous medium equation }

\author[A. Chau]{Albert Chau}
\address{Department of Mathematics, The University of British Columbia, 1984 Mathematics Road, Vancouver, B.C.,  Canada V6T 1Z2.} 
\author[B. Weinkove]{Ben Weinkove}
\address{Department of Mathematics, Northwestern University, 2033 Sheridan Road, Evanston, IL 60208, USA.}

\thanks{Research supported in part by  NSERC grant $\#$327637-06 and NSF grant DMS-2005311.}

\maketitle

\vspace{-20pt}

\begin{abstract}  We prove existence of smooth solutions to linear degenerate parabolic equations on bounded domains assuming  a structure condition of Fichera.  We use this to give a proof of a smooth short time existence result for the porous medium equation $u_t = \Delta u^m$ for $1<m\le 2$.
\end{abstract}

\maketitle

\section{Introduction}

We consider the existence and uniqueness of smooth solutions to a class of degenerate linear parabolic equations.   We revisit an approach introduced by Kohn-Nirenberg \cite{KN} for degenerate elliptic-parabolic equations where a structure condition due to Fichera \cite{Fi56, Fi60} plays a key role (denoted  (B) below). While Kohn-Nirenberg's theorem does not apply directly to our setting, we exploit their methods to prove a general existence result, Theorem \ref{theoremmainintro} below, for linear parabolic equations which degenerate at the boundary assuming the Fichera condition when $n\ge 2$.

There is a vast literature on the subject of degenerate linear parabolic equations, with many different approaches, and we refer the reader to \cite{DH98, DH99, DPT, FP, Fi56, Fi60, JX, KN, Ol, OR} and the references therein.

Let $\Omega_0 \subset \mathbb{R}^n$ be a domain with smooth boundary $\partial \Omega_0$.  We consider the following equation for $w(x,t)$:
\begin{equation} \label{leintro}
\begin{split}
\frac{\partial w}{\partial t} =  {} & a^{ij} w_{ij} + b^i w_i + fw + g, \quad \textrm{on } \ov{\Omega}_0 \times [0,\infty) \\
 \quad w(x,0)={}& 0, \quad \textrm{for } x\in \ov{\Omega}_0, 
 \end{split}
\end{equation}
for  $a^{ij}, b^i, f$ and $g$ smooth functions on $\ov{\Omega}_0 \times [0,\infty)$ with $(a^{ij})$ symmetric and $(a^{ij}) >0$ on the open set $\Omega_0$.  We note two important points, which are related:
\begin{enumerate}
\item[(1)] No boundary data for $w$ is specified on $\partial \Omega_0$.  
\item[(2)] The matrix $(a^{ij})$ may be degenerate (have a nontrivial kernel) on $\partial \Omega_0$.
\end{enumerate}
A basic question which we wish to address is as follows.

\begin{question} \label{mainquestion} Under what conditions on $a^{ij}, b^i, f, g$ does (\ref{leintro}) admit a unique smooth solution on $\ov{\Omega}_0 \times [0,\infty)$?
\end{question}

Note that the matrix $(a^{ij})$ \emph{must} be degenerate in general on $\partial \Omega_0$ for there to be uniqueness, given that no boundary data is specified.  We now discuss our assumptions under which we prove that Question \ref{mainquestion} has a positive answer.

  We consider two assumptions on the matrix $(a^{ij})$ on the boundary.  The first is
\begin{equation} \tag{A$_1$} %\label{assaintro}
a^{ij}\nu_i\nu_j  =0, \quad  \textrm{ on }\partial \Omega_0,
\end{equation}
where $\nu=\nu(x)$ is the outward pointing normal at $x$.  By the Cauchy-Schwarz inequality, (A$_1$) is equivalent to $a^{ij} \nu_i=0$ for all $j$.

The second is that
\begin{equation} \tag{A$_2$} %\label{assaintro}
 (\partial_k a^{ij})\nu_k \nu_{i} \nu_j <0, \quad  \textrm{ on }\partial \Omega_0.
\end{equation}
Note that if (A$_1$) is assumed then since $(a^{ij})>0$ in the interior of $\Omega_0$ we automatically have $(\partial_k a^{ij})\nu_k \nu_{i} \nu_j \le 0$, so the point of (A$_2$) is that we have strict inequality.

Next we consider the Fichera condition
\begin{equation} \tag{B} %\label{assabintro}
(b^i - \partial_j a^{ij})\nu_i \le 0,  \quad  \textrm{ on }\partial \Omega_0.
\end{equation}
We also consider two alternatives to this.  The first is stronger:
\begin{equation} \tag{B'} %\label{assabintro}
(b^i - \partial_j a^{ij})\nu_i < 0,  \quad  \textrm{ on }\partial \Omega_0.
\end{equation}
 In the case of dimension $n=1$ we consider instead
\begin{equation} \tag{B''} % \label{binuiintro}
b^i \nu_i \le 0,\quad  \textrm{ on }\partial \Omega_0,
\end{equation}
which, given (A$_2$), is weaker than (B).

Finally, we make the assumption that
\begin{equation} \tag{C} %\label{assumeg}
D_t^k g(x,0)=0, \quad \textrm{for all } x \in \ov{\Omega}_0, \ k=0,1,2, \ldots.
\end{equation}

We prove the following.

\begin{theorem}\label{theoremmainintro} Suppose that \emph{(A$_1$), (A$_2$),  (B)} and \emph{(C)} hold. Then the equation \eqref{leintro} has a unique solution $w(x, t)\in C^{\infty}(\ov{\Omega}_0\times[0, \infty))$.

If $n=1$, the same is true if we replace \emph{(B)} by the weaker assumption \emph{(B'')}.
\end{theorem}

Moreover, in Theorem \ref{theoremmain} below, we establish estimates on $w$ in terms of the coefficients $a^{ij}$, $b^i$, $f$ and $g$.  A difficulty in proving existence of solutions to degenerate equations  is how to regularize the equation without introducing unwanted boundary conditions.  We exploit  Kohn-Nirenberg's trick of adding an operator of high order $2N$ with boundary conditions on the $N$th and higher normal derivatives to the boundary.  As $N$ tends to infinity, the boundary condition disappears!

\begin{remark} \label{remarkA1} \ 
\begin{enumerate}
\item[(i)] This result is not covered by Kohn-Nirenberg's theorem \cite{KN} (see also \cite{OR}) because the setup does not satisfy their crucial assumption (a).  They  do treat the elliptic version of the above result, with no time variable, under the assumption that $-f$ is very large and positive.   They do not assume (A$_2$), so it is natural to expect that this condition can be removed here too.  We show in the course of the proof that   (A$_2$) can be removed  if (B) is replaced by the stronger assumption (B'). 
\item[(ii)] The assumption (A$_1$) is invariant under smooth coordinate changes in $\mathbb{R}^n$.  If (A$_1$) holds then assumptions 
(A$_2$), (B) and (B') are also invariant under coordinate changes in $\mathbb{R}^n$.  For a proof, see \cite[Theorem 1.1.1]{OR} which covers all cases except for (A$_2$), which can be proved by a similar method.  We also note that if all components of the matrix $(a^{ij})$ vanish on $\partial \Omega_0$ (which occurs in dimension $n=1$) then condition (B'') is also invariant under coordinate changes.  
\item[(iii)] Some model cases of Theorem \ref{theoremmainintro} are dealt with by Daskalopoulos-Hamilton \cite{DH98, DH99}, and we note that in \cite{DH98} condition (B'') is used instead of (B).
\item[(iv)] Condition (C) can be weakened by insisting that the time derivatives of $g$ at $t=0$ vanish only in a neighborhood of $\partial \Omega_0$. 
\end{enumerate}
\end{remark}

Our interest in degenerate linear parabolic equations is motivated by \emph{nonlinear} equations such as the porous medium equation $u_t = \Delta u^m$, where $m>1$ is a fixed constant.  There have been various approaches to understand the existence of smooth short-time solutions of the porous medium equation under a non-degeneracy condition, by Daskalopoulos-Hamilton \cite{DH98}, Koch \cite{Koch} (see also 
Shmarev \cite{Sh}). All are technically difficult and use a variety of different methods.  We expect that these and similar problems should fall into a more general framework.

The porous medium equation is a free boundary problem which, through a
Hanzawa transformation \cite{Hanz}, becomes a degenerate nonlinear equation on a fixed
domain but with no boundary conditions.  One should be able to prove short time existence by simply linearizing this equation and applying the Nash-Moser Inverse Function Theorem.   This reduces the problem to a degenerate linear parabolic equation, of the type described above.  What we found, to our surprise, is that the Fichera condition corresponds to the restriction $m\le 2$, which is why our result on the porous medium equation, Theorem \ref{mainthm} below, has this restriction (except when $n=1$, which is easier).

We now describe our result more precisely.  The porous medium equation 
\begin{equation} \label{PME}
u_t = \Delta (u^m), 
\end{equation}
with initial data
$$u(x,0)=u_0(x)$$
models the density of gas $u(x,t)\ge 0$ diffusing in a porous medium.  Here we take $(x,t) \in \mathbb{R}^n \times [0,\infty)$ and a fixed constant $m>1$.  

Since the equation (\ref{PME}) is degenerate where $u=0$, we interpret it in a weak sense.  While one can construct solutions of the porous medium equation for initial data in $L^1(\mathbb{R}^n)$, in this paper we are concerned only with more regular solutions.  For simplicity we always assume that the initial data $u_0$ is  H\"older continuous and the set $\Omega_0 = \{ u_0>0 \}  \subset \mathbb{R}^n$ is connected and bounded.

By a result of Caffarelli-Friedman \cite{CF80}, there exists a unique  H\"older continuous function $u$ on $\mathbb{R}^n \times [0,\infty)$  with $u_t$ and $\Delta (u^m)$ in $L^1_{\textrm{loc}}(\mathbb{R}^n \times [0,\infty))$ and such that (\ref{PME}) holds almost everywhere.   The function $u$ is smooth on the union of open sets $\Omega_t$ ``where there is gas'', defined by
$$\Omega_t = \{ x\in \mathbb{R}^n \ | \ u(x,t)>0\}.$$
Each $\Omega_t$ is compactly supported in $\mathbb{R}^n$.  We refer the readers to V\'azquez' comprehensive text  \cite{V} for more details and background on the porous medium equation.

It is convenient to work with the \emph{pressure}  $v: = \frac{m}{m-1} u^{m-1}$, satisfying
\begin{equation} \label{PME2}
v_t = (m-1) v \Delta v + |\nabla v|^2,
\end{equation}
on the sets $\Omega_t$ for all $t$ and with initial data $v(x,0) = v_0(x) = \frac{m}{m-1} u_0^{m-1}(x)$.  A short formal calculation shows that, assuming sufficient regularity, the sets $\Omega_t$ move in the direction of the outward normal with speed $|\nabla v|$.  

We prove the following.

\begin{theorem} \label{mainthm} Assume that $1 < m \le 2$ if $n \neq  1$.  Suppose that $v_0 \in C^{\infty} (\ov{\Omega}_0)$ and nondegenerate, meaning that $|\nabla v_0| >0$ on $\partial \Omega_0$.

Let $v=v(x,t)$ be the solution of the porous medium equation (\ref{PME2}).
Then there exists a constant $T>0$ such that the solution $v$ is smooth on the closed set
$$\bigcup_{t \in [0,T]} \ov{\Omega}_t$$
and the boundaries $\partial \Omega_t$ for $t \in [0,T]$ are smooth hypersurfaces in $\mathbb{R}^n$.
\end{theorem}

We make the following remarks.

\begin{remark}   \ 
\begin{enumerate}
\item[(i)]  The nondegeneracy condition $|\nabla v_0| >0$ on $\partial \Omega_0$ means that the boundary initially moves with nonzero speed at all points, and is relevant for example in studies of concavity of solutions \cite{ChW, DHL, IS}.  Note also that $\partial \Omega_0$ is smooth by our assumptions.  By shrinking $T>0$ if necessary, we may conclude that $v$ is nondegenerate for $t \in [0,T]$.
\item[(ii)] In the above, the existence of a weak solution to the porous medium equation is assumed only for convenience.  The proof of Theorem \ref{mainthm} also produces a solution for a short time.
\item[(iii)] If  Theorem \ref{theoremmainintro} holds with (A$_1$) replaced by the assumption that the components $a^{ij}$ all vanish on $\partial \Omega_0$ and (B) replaced by (B'') then the restriction $m\le 2$ could be removed.
\item[(iv)] Daskalopoulos-Hamilton \cite{DH98} and Koch's Habilitation \cite{Koch} do not have the restriction $m\le 2$.  They assume less regularity at $t=0$ and show smoothness  for $t>0$ (they do not address smoothness of the solution going back to $t=0$).  For earlier work on regularity, see \cite{CVW, CW}.  In dimension $n=1$, the result of  Theorem \ref{mainthm} was proved by Aronson-V\'azquez \cite{AV} and H\"ollig-Kreiss \cite{HK} independently (see also \cite{An}).  All of these references use different methods from ours.
\item[(v)] Physically, the case $m=2$ is particularly important, coinciding with Boussinesq's equation for groundwater flow.  In terms of flows of gases, the case $m=2$ arises as a model for isothermal processes, whereas $m>2$ corresponds to adiabatic processes.  We refer the reader to the discussion and references in \cite[Chapter 2]{V}, for example. 
\item[(vi)] Similar methods can deal with some cases of the degenerate Gauss curvature flow (a ``rolling stone with a flat side'', cf.  \cite{DH99}). The authors will treat this in a follow-up work.

\end{enumerate}
\end{remark}

The outline of the paper is as follows.  In Section  \ref{Solving}, we prove Theorem \ref{theoremmainintro} on the linear degenerate equation.  In Section \ref{sectiontrans}, we write the porous medium equation as a nonlinear equation on a fixed domain, with no boundary conditions and show how the Fichera condition for the linearized equation corresponds to $m \le 2$.  Finally we prove Theorem \ref{mainthm} in Section \ref{sectionlast} using the Nash-Moser Inverse Function Theorem.

\bigskip
\noindent
{\bf Acknowledgements.} \ The second-named author thanks Jared Wunsch for a helpful conversation.

\section{Solving the linear equation} \label{Solving}

Let $\Omega_0 \subset \mathbb{R}^n$ be a domain with smooth boundary $\partial \Omega_0$.  Fix $T$ with $0<T<\infty$. We wish to construct a smooth solution $w(x,t)$  to the equation
\begin{equation} \label{le}
\begin{split}
\frac{\partial w}{\partial t} =  {} & a^{ij} w_{ij} + b^i w_i + fw + g, \quad \textrm{on } \ov{\Omega}_0 \times [0,T] \\
 \quad w(x,0)={}& 0, \quad \textrm{for } x\in \ov{\Omega}_0,
 \end{split}
\end{equation}
for $a^{ij}, b^i, f$ and $g$ smooth functions on $\ov{\Omega}_0 \times [0,T]$ and $(a^{ij})$ symmetric. 

We will prove:

\begin{theorem}\label{theoremmain}Suppose that \emph{(A$_1$), (A$_2$),  (B)} and \emph{(C)} hold.
Then the equation \eqref{le} has a unique solution $w(x, t)\in C^{\infty}(\ov{\Omega}_0\times[0, T])$.  Moreover, there exist an integer $r>0$ depending only the dimension $n$ and constants $C(\ell)$ depending only on $n, \ell, T, \Omega_0$ and spatial $C^r$ bounds of $a^{ij}, b^i, f$ such that for $\ell=0,1,2,\ldots,$
	\begin{equation}\label{tameest}
		\|w(t)\|_{C^\ell} \leq C(\ell) \sup_{t \in [0,T]}\bigg( 1+ \sum_{i,j} \|a^{ij}\|_{C^{\ell+r}} + \sum_i \|b^i\|_{C^{\ell+r}}+ \| f\|_{C^{\ell+r}} + \|g\|_{C^{\ell+r}}\bigg),
	\end{equation}
where $\| \cdot \|_{C^\ell}$ denotes the usual $C^{\ell}(\ov{\Omega}_0)$ norm.

The same holds if $n=1$ and assumption \emph{(B)} is replaced by \emph{(B'')}.
	\end{theorem}

Since $T$ is arbitrary, Theorem \ref{theoremmain} immediately implies Theorem \ref{theoremmainintro}.  Note that the solution $w(x,t)$ of Theorem \ref{theoremmain} satisfies
\begin{equation} \label{wconclusion}
D_t^k w(x,0)=0, \quad \textrm{for all } x \in \ov{\Omega}_0, \ k=0,1,2, \ldots.
\end{equation}

\medskip

First we have the following uniqueness statement.  Similar results can be found in \cite{OR}, for example, but we include here the proof for the sake of completeness.

\begin{proposition} \label{propunique}
Assuming \emph{(A$_1$)}, \emph{(B)} and \emph{(C)}, let $w_1, w_2$ be solutions to (\ref{le}) which are $C^2$ in space and $C^1$ in time.  Then $w_1=w_2$.  

The same is true in dimension $n=1$ if we replace \emph{(B)} by \emph{(B'')}.
\end{proposition}
\begin{proof}
Define $\tilde{w}:=w_1-w_2$, which satisfies,
\begin{equation} \label{le2}
\begin{split}
\frac{\partial \tilde{w}}{\partial t} =  {} & a^{ij} \tilde{w}_{ij} + b^i \tilde{w}_i + f\tilde{w}, \quad \textrm{on } \ov{\Omega}_0 \times [0,T] \\
 \quad \tilde{w}(x,0)={}& 0, \quad \textrm{for } x\in \ov{\Omega}_0.
 \end{split}
\end{equation}
Set $v= e^{-Ct}\tilde{w} -\delta t$ for $C = \max_{\ov{\Omega}_0 \times [0,T]} f$ and a small constant $\delta>0$.  
Suppose for a contradiction that $v$ achieves a positive maximum on $\ov{\Omega}_0 \times [0,t_0]$ at the point $(x_0, t_0)$ with $t_0>0$.  Then if $x_0$ is an interior point of $\Omega_0$ we have $(v_{ij}) \le 0$ and $v_i=0$ at that point, giving
$$0 \le \frac{\partial v}{\partial t} = a^{ij} v_{ij} + b^i v_i + (f-C)(v+\delta t) - \delta \le -\delta <0,$$
a contradiction.

Now suppose that $x_0$ is a point on the boundary $\partial \Omega_0$.  Suppose that $\partial \Omega_0$ near $x_0$ can be written as a graph $x_n = \gamma(x')$  with $\gamma(0)=0$ so that  $x_0$ represents the origin, writing $x'=(x_1, \ldots, x_{n-1})$.  We may assume that $D_{x'} \gamma(0)=0$ and the unit normal at $x_0$ is $(0,\ldots, 0,1)$. 
For $j=1, \ldots, n-1$, we apply $\partial_j$ to $a^{ij} \nu_i=0$ to obtain at $0$,
$$\sum_{i,j=1}^{n-1} a^{ij} \nu_{ij}  + \sum_{j=1}^{n-1} \partial_j a^{nj}=0,$$
where we regard $\nu_i$ as a function of $x_1, \ldots, x_{n-1}$, given by
$$\nu_i = - \frac{1}{\sqrt{1+|D_{x'} \gamma|^2}} \gamma_i, \quad \textrm{for } i=1, \ldots, n-1.$$
Hence, at the origin,
$$0 \ge (b^i - \partial_j a^{ij} )\nu_i = b^n - \partial_n a^{nn} + \sum_{i,j=1}^{n-1} a^{ij} \nu_{ij} = b^n - \partial_n a^{nn} - \sum_{i,j=1}^{n-1} a^{ij} \gamma_{ij},$$
and since $\partial_n a^{nn} \le 0$ we have
\begin{equation}\label{crucial}
b^n - \sum_{i,j=1}^{n-1} a^{ij} \gamma_{ij} \le 0.
\end{equation}

Since
 $v(x_0)$ is a local maximum for $x \in \partial \Omega_0$, applying $\partial_i$, $\partial_j$ to  $x' \mapsto v(x', \gamma(x'))$ for $i,j=1, \ldots, n-1$,
 $$0 \ge (v_{ij} + v_n \gamma_{ij}), \quad \textrm{at } x_0,$$
 as $(n-1)\times (n-1)$ matrices. Hence $$\sum\limits_{i,j=1}^{n}a^{ij}v_{ij} + b^i v_i \le \sum\limits_{i,j=1}^{n-1} - a^{ij} \gamma_{ij} v_n  + b^n v_n \le 0,$$
 using (\ref{crucial}), the fact that $a^{in}=0$ for all $i$ by the remark under (A$_1$), and the fact that $v_n \ge 0$ and $v_1=\ldots  =v_{n-1}=0$ at $x_0$.
Then
$$0 \le \frac{\partial v}{\partial t} = a^{ij} v_{ij} + b^i v_i + (f-C)(v+\delta t) - \delta \le -\delta <0,$$
a contradiction.  This proves that $v \le 0$ on $\ov{\Omega}_0 \times [0,T]$.  Letting $\delta$ tend to zero gives $\tilde{w} \le 0$.  The proof of  $\tilde{w} \ge 0$ is similar.
\end{proof}

\begin{remark}
An easier argument than the one above shows that the same uniqueness statement holds if (A$_1$) is replaced by the assumption that all the components $a^{ij}$ vanish on $\partial \Omega_0$ and (B)
is replaced by (B'') (cf. \cite{DH98}).  This includes the case of $n=1$, with (B) replaced by (B'').
\end{remark}

To prove Theorem \ref{theoremmain}, we regularize our equation using a trick of 
Kohn-Nirenberg \cite{KN}.  Fix a large integer $N$ and regularize with an elliptic term of order $2N$ which we will denote by $L_{2N}$.  For convenience, we  assume that $N$ is an \emph{even} integer.  Let $\ve>0$ be small, and consider the following parabolic equation for $w(x,t)$:
\begin{equation} \label{ler}
\begin{split}
\frac{\partial w}{\partial t} =  {} & -\ve L_{2N} w+ a^{ij} w_{ij} + b^i w_i + fw + g, \quad \textrm{on } \ov{\Omega}_0 \times [0,T] \\
 \quad w(x,0)={}& 0, \quad \textrm{for } x\in \ov{\Omega}_0\\
 \partial_{\nu}^{j} w(x, t)= {} & 0, \quad \textrm{for } (x, t) \in \partial \Omega_0 \times [0,T], \quad j=N, N+1, \ldots, 2N-1.
 \end{split}
\end{equation}

We now describe the operator $L_{2N}$, which we will define locally and then patch together.  We cover $\ov{\Omega}_0$ with a finite number of relatively open charts $\{ U_{\lambda} \}_{\lambda=0,1, \ldots, K}$ such that  each $U_{\lambda}$, for $\lambda \ge 1$, contains a piece of the boundary $\partial \Omega_0$ and their union consists of points in $\ov{\Omega}_0$ at most a distance $c_0>0$ from the boundary, where we may shrink $c_0>0$ if necessary.  

Next, for $c>0$, we write $\Omega_c$ to be the set of points of $\Omega_0$ whose distance from the boundary is strictly greater than $c$.  Define $U_0 = \Omega_{c_0/4}$.

As in \cite[Section 3.2]{KN} we choose special coordinates $(X_1^{\lambda},\ldots, X_{n-1}^{\lambda}, Y^{\lambda}) \in Q_{n-1} \times (-c_0, 0]$ on each $U_{\lambda}$ for $\lambda =1, \ldots, K$ such that $-Y^{\lambda}$ equals the distance from the boundary.  On the surfaces $Y^{\lambda}=\textrm{constant}$ we use $X_1^{\lambda}, \ldots, X_{n-1}^{\lambda}$ as local coordinates.  Here we are writing $$Q_{n-1} := (-1,1)^{n-1}.$$
The boundary 
$\partial \Omega_0 \cap U_{\lambda}$ corresponds to the points  with $Y^{\lambda}=0$.  The outward normal $\nu$ along the boundary in these coordinates is the unit vector $(0, \ldots, 0, 1)$.  Denote by $\Phi_{\lambda}$ the diffeomorphisms from $U_{\lambda}$ to  $Q_{n-1} \times (-c_0,0]$.

Let $\{ \xi_{\lambda} \}_{\lambda=0}^K$ be a partition of unity subordinate to this covering.  Namely, we have smooth functions $\xi_{\lambda}: U_{\lambda} \rightarrow [0,1]$  supported compactly in these subsets (which allows $\xi_{\lambda}$ to be nonzero for some of the boundary part $Y^{\lambda}=0$) and so that
$$\sum_{\lambda=0}^K \xi_{\lambda}=1,$$
on $\ov{\Omega}_0$.  We also assume that
\begin{equation} \label{DYx}
D_{Y^{\lambda}} \xi_{\lambda} =0, \quad \textrm{for } |Y^{\lambda}| \le c_0/2, \quad \lambda =1, \ldots, K.
\end{equation}

Define $L_{2N}$ by
\begin{equation} \label{L2N}
\begin{split}
L_{2N} w := {} & \sum_{\lambda=1}^K \xi^2_{\lambda} \left(D^{2N}_{Y^{\lambda}} w +  \sum_{|\bal | = N}  D^{2\bal}_{X^{\lambda}} w    \right)  + \xi_0^2 \sum_{|\mathbf{a}|=N} D_{x}^{2\mathbf{a}} w
\end{split}
\end{equation}
where for the sake of simplicity we are suppressing the diffeomorphisms mapping $Q_{n-1} \times (-c_0,1]$ to $U_{\lambda}$.  Here $\bal=(\al_1, \ldots, \al_{n-1})$ and $2\bal=(2\al_1, \ldots, 2\al_{n-1})$ are multi-indices with $n-1$ components, while $\mathbf{a}$ and $2\mathbf{a}$ have $n$ components.

\begin{lemma}
There exists a unique smooth solution $w=w(x,t)$ to (\ref{ler}) on $\ov{\Omega}_0\times [0,T]$.
\end{lemma}
\begin{proof}
This holds by standard theory for higher order parabolic equations.  Indeed we may apply \cite[Theorem II.10.3]{F}.  Here we use the version of this result for general boundary conditions (see  \cite[p. 142]{F}).
Write $A_0$ for the principal symbol of $L_{2N}$.  At $x \in \partial \Omega_0$, the polynomials $p(z) = A_0 (x,\xi + z \nu) - re^{i\theta}$ have precisely $N$ roots with positive imaginary parts, where $r>0$ and $\theta \in [\pi/2, 3\pi/2]$ are fixed constants, $\nu$ is the outer unit normal, and $\xi$ is a real nonzero vector in the tangent hyperplane to $\partial \Omega_0$ at $x_0$.  By the discussion in \cite[p. 76-77]{F}, this is sufficient to apply \cite[Theorem II.10.3]{F} in our setting.  \end{proof}

In the following we will make use of this lemma.

\begin{lemma}  \label{lemmagarding} Assume that $q$ is a positive integer.
Let $u$ be a smooth real-valued function on $Q_{n-1} \times (-c_0,0]$ supported on a compact set in $Q_{n-1} \times (-c_0,0]$.  
Then for every $\delta>0$ there is a constant $C_{\delta}$ such that 
\begin{equation} \label{garding}
\begin{split}
\lefteqn{\int_{Q_{n-1}} \int_{-c_0}^0 \sum_{k=0}^{q-1} \sum_{|\bal | =q - k} | D_Y^k D_X^{\bal}u|^2dYdX  } \\ \le {} &  \delta  \int_{Q_{n-1}} \int_{-c_0}^0 |D_Y^q u|^2 dYdX + C_{\delta}\int_{Q_{n-1}} \int_{-c_0}^0 \sum_{|\bal|=q} |D^{\bal}_X u|^2 dYdX,
\end{split}
\end{equation}
where we use coordinates $X=(X_1, \ldots, X_{n-1})$ on $Q_{n-1}$ and $Y$ on $(-c_0,0]$.  
In addition, for any $k$,
\begin{equation} \label{knY}
\int_{Q_{n-1}} \int_{-c_0}^0 | D_Y^k u|^2dYdX \le C \int_{Q_{n-1}} \int_{-c_0}^0 (-Y)^2 | D_Y^{k+1} u|^2dYdX.
\end{equation}
\end{lemma}
\begin{proof}
The inequality (\ref{garding}) is proved in \cite[Lemma 2.4]{MN}.  For (\ref{knY}), see \cite[Lemma 5.2]{KN}.
\end{proof}

We will also use the following basic results.  We define as usual the Sobolev norm
$$\| u \|^2_{H^q} := \sum_{|\mathbf{a}| \le q} \int_{\Omega_0} |D^{\mathbf{a}} u|^2 dx,$$
 and as above denote by 
$\| u \|_{C^r}$ the $C^r(\ov{\Omega}_0)$ norm of $u$.

\begin{theorem} \label{thmbasic} Let $u, v$ be in  $C^{\infty}(\ov{\Omega}_0)$ and let $q$ be a positive integer.
\begin{enumerate}
\item[(i)] (Sobolev Embedding) If $0 \le r< q-n/2$ then 
$$\| u \|_{C^r} \le C \| u \|_{H^q}.$$
\item[(ii)] For every $\delta>0$ there exists $C_{\delta}$ depending only on $n$, $\Omega_0$ and $q$ such that 
$$\| u \|_{H^{q-1}}^2  \le C \left( \int_{\Omega_0}u^2 +  \int_{\Omega_0} \sum_{|\mathbf{a}|=q-1} |D^{\mathbf{a}}u|^2 \right)  \le C_{\delta}  \int_{\Omega_0}u^2 + \delta \int_{\Omega_0} \sum_{|\mathbf{a}|=q} |D^{\mathbf{a}}u|^2.$$
\item[(iii)] For $0 \le \ell \le k < q-\ell$,
$$\| u\|_{C^k} \| v\|_{C^{q-k}} \le C ( \| u \|_{C^\ell} \| v\|_{C^{q-\ell}} + \| u\|_{C^{q-\ell}} \| v \|_{C^{\ell}}).$$
\end{enumerate}
\end{theorem}
\begin{proof} Part (i) is completely standard.  Part (ii) is a consequence of the Gagliardo-Nirenberg interpolation inequalities (see \cite[p. 125-126]{N}).
For (iii), see \cite[Corollary 2.2.2]{H}.
\end{proof}

\medskip

Define for a positive integer $k$,
\[
\begin{split}
I_k(t) = {} & \int_{\Omega_0} w^2 dx+\sum_{\lambda=1}^K \int_{Q_{n-1}} \int_{-c_0}^0  \left(   (-Y_{\lambda}) (D_{Y^{\lambda}}^k w)^2 + \sum_{|\bal| = k}(D_{X^{\lambda}}^{\bal} w)^2 \right) \xi^2_{\lambda}  dY^{\lambda}dX^{\lambda} \\
{} & + \int_{U_0} \sum_{ |\mathbf{a}| = k} (D_x^{\mathbf{a}} w)^2 \xi_0^2 dx,
\end{split}
\]
where $\{\xi_{\lambda}\}$ is a partition of unity as described above and $w=w(x,t)$ is the solution of (\ref{ler}). Note that we are suppressing the diffeomorphisms mapping the charts $U_{\lambda}$ to $Q_{n-1} \times (-c_0,0]$.  We have
$I_k (t) \le C \| w \|_{H^k}^2$
and
 from Lemma \ref{lemmagarding} and part (ii) of Theorem \ref{thmbasic},
\begin{equation} \label{equivalence}
\begin{split}
 \| w \|_{H^{k-1}}^2 + \| w \|_{H^k(\Omega_{c_0/4})}^2 + \sum_{\lambda=1}^K \int_{Q_{n-1}} \int_{-c_0}^0 \sum_{|\bal|\le k} (D^{\bal}_{X^{\lambda}}w)^2 \xi_{\lambda}^2 dYdX
  \le  C I_k(t),
\end{split}
\end{equation}
for a uniform constant $C$ depending only on $n$, $k$, $\Omega_0$, $U_{\lambda}$, $\xi_{\lambda}$.  Note that to bound the third term on the left hand side, we apply part (ii) of Theorem \ref{thmbasic}, not to $\Omega_0$, but instead fixing $Y^{\lambda}$ and considering $w$ as a function of $X^{\lambda}$ alone.  

The following key lemma is the heart of this section.  

\begin{lemma} \label{klemma} For sufficiently large $N$ the following holds.  There is a positive integer $r$ depending only on $n$ and constant $C$ depending only on $n$, $N$, $T$, $\Omega_0$ and spatial $C^r$ bounds of $a^{ij}, b^i, f$   such that if $w(x,t)$ solves (\ref{ler}) then
\begin{equation} \label{kle}
\frac{d}{dt} I_{N} \le  - \frac{\ve}{C} I_{2N} + C I_{N}+ C^2_N (I_r+1),
\end{equation}
where 
$$C_N:= \sup_{t\in [0,T]} \bigg(1+\sum_{i,j} \|a^{ij}\|_{C^{N+r}} + \sum_i \|b^i\|_{C^{N+r}}+ \| f\|_{C^{N+r}} + \|g\|_{C^{N+r}} \bigg).$$
\end{lemma}

\begin{proof}[Proof of Lemma \ref{klemma}]
In what follows, we will denote by $C$ a uniform constant as in the assertion of the lemma.  We first claim that
\begin{equation} \label{intws}
\frac{d}{dt} \int_{\Omega_0} w^2 dx \le   C(I_{r}+1),
\end{equation}
for any integer $r$ with $r \ge n/2+2$.

Compute
\begin{equation} \label{0}
\begin{split}
\frac{d}{dt} \int_{\Omega_0} w^2 dx = -2\ve \int_{\Omega_0} w L_{2N} w  dx + 2 \int_{\Omega_0} w \left( a^{ij} w_{ij} + b^i w_i +fw+g \right) dx.
\end{split}
\end{equation}
Recalling condition (A$_1$), we integrate by parts to obtain
\begin{equation} \label{1}
\begin{split}
\int_{\Omega_0} 2w (a^{ij} w_{ij} + b^i w_i) dx = {} & \int_{\Omega_0} (-2a^{ij} w_i w_j + (-a^{ij}_j + b^i) (w^2)_i) dx \\
= {} & \int_{\Omega_0} (-2a^{ij} w_i w_j + (a^{ij}_j - b^i)_i (w^2) ) dx + \int_{\partial \Omega_0} \nu_i (b^i  -a^{ji}_j) w^2 dx\\
\le {} & C \int_{\Omega_0} w^2dx +C \| w \|^2_{H^{r-1}} \\
\le {} & CI_{r},
\end{split}
\end{equation}
for any integer $r$ with $r \ge n/2+2$,
where we used part (i) of Theorem \ref{thmbasic} and (\ref{equivalence}).  Observe that in the case $n>1$, where (B) is assumed, the right hand side of (\ref{1}) can be replaced by $C \int_{\Omega_0} w^2dx.$

We also have
\begin{equation} \label{2}
\begin{split}
\int_{\Omega_0} 2w (fw+g) dx 
\le {} & C \int_{\Omega_0} (w^2 + g^2) dx.
\end{split}
\end{equation}
Next, recalling  (\ref{L2N})  we have
\[    
\begin{split}
-2\ve \int_{\Omega_0} w L_{2N} w dx = {} & -2\ve \sum_{\lambda=1}^K \int_{Q_{n-1}} \int_{-c_0}^0 \xi^2_{\lambda} w \left( D^{2N}_{Y^{\lambda}} w + \sum_{|\bal | = N} D^{2\bal}_{X^{\lambda}} w  \right) V_{\lambda} dY^{\lambda} dX^{\lambda} \\
{} & -2\ve \int_{\Omega_0} \xi_0^2  w \sum_{|\mathbf{a}|=N} D_{x}^{2\mathbf{a}} w dx,
\end{split}
\]
where $V_{\lambda}$ are smooth positive functions arising from the change of variables.   Integrating by parts we obtain
\begin{equation} \label{3}
\begin{split}
\lefteqn{
-2\ve\int_{\Omega_0} w L_{2N} w dx } \\ \le {} &   -\ve \sum_{\lambda=1}^K \int_{Q_{n-1}} \int_{-c_0}^0 \left( (D_{Y^{\lambda}}^N  w)^2 + \sum_{|\bal| = N} ( D^{\bal}_{X^{\lambda}} w )^2 \right) V_{\lambda} \xi_{\lambda}^2 dY^{\lambda} dX'^{\lambda}  \\
{} & -\ve \int_{\Omega_0} \xi_0^2\sum_{|\mathbf{a}|=N}  (D_{x}^{\mathbf{a}} w)^2dx + C \ve\| w\|^2_{H^{N-1}} \\
\le {} &  - \frac{\ve}{C} \| w\|^2_{H^N} + C \ve \delta \sum_{|\mathbf{a}| =N} \int_{\Omega_0} |D^{\mathbf{a}}w|^2 dx  + C_{\delta} \ve \int_{\Omega_0} w^2dx \\
\le {} &  C_{\delta} \ve \int_{\Omega_0} w^2dx,
\end{split}
\end{equation}
where we used Lemma \ref{lemmagarding} and part (ii) of Theorem \ref{thmbasic} and chose $\delta>0$ sufficiently (uniformly) small.  When integrating by parts in $Y^{\lambda}$, we use the fact that $D^N_{Y^{\lambda}}w, \ldots, D^{2N-1}_{Y^{\lambda}} w$ vanish on the boundary $Y^{\lambda}=0$.  When we integrate by parts in $X^{\lambda}_1, \ldots, X^{\lambda}_{n-1}$ there are no boundary terms since $\xi_{\lambda}$ vanishes at $X^{\lambda}_{\ell} = \pm 1$. 

Combining (\ref{0}), (\ref{1}), (\ref{2}) and (\ref{3}) gives (\ref{intws}).

Next, it is convenient to work in our special coordinate charts as described above.   We assume now that $n \ge 2$, since the case of $n=1$  is simpler and will be addressed at the end of the proof.  
Fix an $\lambda$ between $1$ and $K$ and work in the fixed chart $U_{\lambda}$.  When we change coordinates, the coefficients of the second order operator $a^{ij} w_{ij} + b^i w_i + fw$ change.  Recalling Remark \ref{remarkA1}, (ii), conditions (A$_1$), (A$_2$),  (B) and (B') become 
\begin{equation} \tag{A$_1$}
a^{in}=0, \ i=1, \ldots, n, \quad   \textrm{on } Y^{\lambda}=0,
\end{equation}
\begin{equation} \tag{A$_2$}
 \partial_n a^{nn}<0, \quad  \textrm{on } Y^{\lambda}=0,
 \end{equation}
\begin{equation} \tag{B}
 \quad  b^n - \partial_j a^{nj} \le 0,  \quad \textrm{on } Y^{\lambda}=0,
\end{equation}
\begin{equation} \tag{B'}
 \quad  b^n - \partial_j a^{nj} < 0,  \quad \textrm{on } Y^{\lambda}=0.
\end{equation}
Note that $D_{Y^{\lambda}}$ and $D_{X_i^{\lambda}}$, for $i=1, \ldots, n-1$, are the same as $\partial_n$ and $\partial_i$, and we will use these two notations interchangeably.

Define for $\lambda=1, \ldots, K$,
$$I_{N, \lambda}(t) := \int_{Q_{n-1}} \int_{-c_0}^0  \left(   (-Y^{\lambda}) (D_{Y^{\lambda}}^N w)^2 + \sum_{|\bal| = N}(D_{X^{\lambda}}^{\bal} w)^2 \right) \xi^2_{\lambda}  dY^{\lambda}dX^{\lambda}.$$
Then
\[
\begin{split}
\lefteqn{\frac{d}{dt} I_{N, \lambda} } \\= {} & -2 \ve  \int_{Q_{n-1}} \int_{-c_0}^0 \left( -Y^{\lambda} (D_{Y^{\lambda}}^{N} w) D^{N}_{Y^{\lambda}} L_{2N} w +  \sum_{|\bal| = N} (D^{\bal}_{X^{\lambda}} w) D^{\bal}_{X^{\lambda}} L_{2N} w \right) \xi^2_{\lambda} dY^{\lambda} dX^{\lambda} \\
{} & + \int_{Q_{n-1}} \int_{-c_0}^0 2 (-Y^{\lambda})(D_{Y^{\lambda}}^{N} w) D_{Y^{\lambda}}^{N} (a^{ij} w_{ij} + b^i w_i + f w + g)  \xi^2_{\lambda} dY^{\lambda} dX^{\lambda} \\
{} & + \int_{Q_{n-1}} \int_{-c_0}^0 \sum_{|\bal| = N} 2 (D_{X^{\lambda}}^{\bal} w) D_{X^{\lambda}}^{\bal} (a^{ij} w_{ij} + b^i w_i + f w + g)  \xi^2_{\lambda} dY^{\lambda} dX^{\lambda}  \\
  =: {} & \fbox{1} + \fbox{2} + \fbox{3}.
\end{split}
\]
For $\fbox{1}$,  carry out $N$ integrations by parts to obtain
\begin{equation} \label{equationbox1}
\begin{split}
\fbox{1} \le {} & - \ve \int_{Q_{n-1}} \int_{-c_0}^0 \bigg\{ \sum_{\mu =1}^K \xi_{\lambda}^2 \xi_{\mu}^2 \bigg( -Y^{\lambda} (D_{Y^{\lambda}}^N D_{Y^{\mu}}^N w)^2  + \sum_{ |\bal| = N} (-Y^{\lambda}) (D_{Y^{\lambda}}^N D_{X^{\mu}}^{\bal} w)^2  \\
{} & + \sum_{ |\bal| = N} (D_{X^{\lambda}}^{\bal} D^N_{Y^{\mu}} w)^2 + \sum_{ |\bal| = N} \sum_{ |\bbe| = N} (D^{\bal}_{X^{\lambda}} D^{\bbe}_{X^{\mu}} w)^2 \bigg)  \\
{} & + \sum_{|\mathbf{a}|=N} (-Y^{\lambda})(D_{Y^{\lambda}}^N D_{x}^{\mathbf{a}} w)^2 \xi_0^2 \xi_{\lambda}^2 + \sum_{ |\bal| = N} \sum_{|\mathbf{a}|=N} (D_{X^{\lambda}}^{\bal} D_{x}^{\mathbf{a}} w)^2 \xi_0^2 \xi_{\lambda}^2 \bigg\} dY^{\lambda} dX^{\lambda} \\
{} & + \frac{\ve}{C'} I_{2N} + C' \ve\int_{\Omega_0}w^2 dx, 
\end{split}
\end{equation}
for a large constant $C'$ which we can increase if necessary.  The above needs some justification, since there are lower order terms arising when derivatives land on $-Y^{\lambda}$ or cut-off functions, or functions arising from the change of variables.  First we consider the term where a $Y^{\mu}$ derivative lands on the term $(-Y^{\lambda})$.  Note that since $-Y$ denotes the distance from the boundary, it is the same variable in each coordinate patch.  We have a term like
$$2\ve \int_{Q_{n-1}} \int_{-c_0}^0 N (D_{Y^{\lambda}}^N D_{Y^{\mu}}^{N-1} w)(D_{Y^{\lambda}}^N D_{Y^{\mu}}^{N} w)\xi_{\lambda}^2 \xi_{\mu}^2dY^{\lambda} dX^{\lambda},$$
which we may write as
$$\ve \int_{Q_{n-1}} \int_{-c_0}^0 N \left( D_{Y^{\mu}} (D_{Y^{\lambda}}^N D_{Y^{\mu}}^{N-1} w)^2 \right) \xi_{\lambda}^2 \xi_{\mu}^2dY^{\lambda} dX^{\lambda},$$
and integrating $D_{Y^{\mu}}$ by parts and using (\ref{DYx}) gives a term of order $\ve \| w \|^2_{H^{2N-1}(\Omega_{c_0/2})}$, which by \eqref{equivalence} can be bounded by $\frac{\ve}{C'} I_{2N} + C' \ve\int_{\Omega_0}w^2 dx$.

Next consider a term where a $D_X$ derivative lands on a cut-off function.  One such term is of the form 
$$- 2\ve \int_{Q_{n-1}} \int_{-c_0}^0 \xi_{\mu}^2 (D_{X_i^\mu} \xi_{\lambda}^2)  (D^{\bal}_{X^{\lambda}} D^{\bbe}_{X^{\mu}} w)  (D^{\bal}_{X^{\lambda}} D_{X_i^{\mu}} D^{\bbe}_{X^{\mu}} w)dY^{\lambda}dX^{\lambda},$$ 
with $|\bal|=N$ and $|\bbe |=N-1$.  We can control this by
$$  \frac{\ve}{C} \int_{Q_{n-1}} \int_{-c_0}^0 \xi_{\mu}^2 \xi_{\lambda}^2 (D^{\bal}_{X^{\lambda}} D_{X_i^{\mu}} D^{\bbe}_{X^{\mu}} w)^2dY^{\lambda}dX^{\lambda} + C' \ve \int_{Q_{n-1}} \int_{-c_0}^0 \xi_{\mu}^2 (D^{\bal}_{X^{\lambda}} D^{\bbe}_{X^{\mu}} w)^2dY^{\lambda}dX^{\lambda}.$$
Note that the number of derivatives in the second term is only $2N-1$ so we can bound it  by $\frac{\ve}{C'} I_{2N} + C' \ve\int_{\Omega_0}w^2 dx$.  Similar arguments using \eqref{equivalence} and Lemma \ref{lemmagarding} can be used to bound other lower order terms.

For the rest of the proof, we will drop the $\lambda$ superscripts, for ease of notation.  
We estimate, recalling that  $D_Y^{N}w=0$ on $Y=0$,
\begin{equation} \label{3estimate0}
\begin{split}
\fbox{2} \le {}&  \int_{Q_{n-1}} \int_{-c_0}^0 2(-Y) (D_Y^N w) \bigg[ a^{ij} D_Y^N w_{ij} + N (D_Y a^{ij})( D_Y^{N-1} w_{ij}) + b^i D_Y^N w_i  \bigg] \xi^2 dYdX\\ 
{}&  + C  \int_{Q_{n-1}} \int_{-c_0}^0 |Y|| D_Y^N w| \sum_{\ell=r_0}^{N-r_0}   \bigg(  |D_Y^{\ell} a^{ij}| \, |D_Y^{N-\ell} w_{ij}|  + |D_Y^{\ell} b^i | \,  |D_Y^{N-\ell}w_i| \\ {} & + |D_Y^{\ell} f| \,  |D_Y^{N-\ell} w| \bigg) \xi^2 dYdX + \delta \int_{Q_{n-1}} \int_{-c_0}^0 (D_Y^{N} w)^2 \xi^2 dYdX \\ {} &  + C_{\delta} I_N  +C_N^2 (I_{r_0+2} +1),
\end{split}
\end{equation}
for $r_0=r_0(n)$ chosen so that $r_0 >n/2+2$ and  for a small $\delta>0$ to be determined later.  Here we are using Lemma \ref{lemmagarding} to bound 
\begin{equation} \label{uselg}
\begin{split}
\int_{Q_{n-1}} \int_{-c_0}^0 (-Y) (D_Y^{k} D^{\bal}_X w)^2 \xi^2 dYdX \le {} & \int_{Q_{n-1}} \int_{-c_0}^0 (D_Y^{k} D^{\bal}_X w)^2 \xi^2 dYdX \\
\le{} & \delta \int_{Q_{n-1}} \int_{-c_0}^0 (D_Y^{N} w)^2 \xi^2 dYdX + C_{\delta} I_N,
\end{split}
\end{equation}
whenever  $k+|\bal| = N$ and $|\bal| =1, 2$. Here we assume $c_0\le 1$ without loss of generality.  We will use (\ref{uselg}) several times in the computations that follow.

Integrating by parts the first two terms in the square brackets of (\ref{3estimate0}), we have
\begin{equation} \label{3estimate}
\begin{split}
\fbox{2} \le {}&  \int_{Q_{n-1}} \int_{-c_0}^0 \bigg[ -2(-Y)a^{ij} (D_Y^N w_i) (D_Y^N w_j) +  (-Y)(b^i-a^{ij}_j) ((D_Y^N w)^2)_i \\
{} & + 2a^{in} (D_Y^N w)(D_Y^N w_i) 
 - 2(-Y)N (D_Y a^{ij}) (D_Y^N w_j) (D_Y^{N-1} w_i)  \\
 {} & + 2N (D_Y^N w) (D_Y a^{in}) (D_Y^{N-1} w_i)  
 \bigg] 
 \xi^2 dYdX \\ {} & + C  \int_{Q_{n-1}} \int_{-c_0}^0 |Y|| D_Y^N w| \sum_{\ell=r_0}^{N-r_0}   \bigg(  |D_Y^{\ell} a^{ij}| \, |D_Y^{N-\ell} w_{ij}|  + |D_Y^{\ell} b^i | \,  |D_Y^{N-\ell}w_i| \\ {} & + |D_Y^{\ell} f| \,  |D_Y^{N-\ell} w| \bigg) \xi^2 dYdX  - \int_{Q_{n-1}} \int_{-c_0}^0 4(-Y)a^{ij} \xi \xi_j (D_Y^Nw) (D_Y^N w_i) dY dX \\
 \\ {} & + \delta \int_{Q_{n-1}} \int_{-c_0}^0 (D_Y^{N} w)^2 \xi^2 dYdX  + C_{\delta} I_N  +C_N^2 (I_{r_0+2} +1).
\end{split}
\end{equation}

We estimate the terms above one by one.  The first term in the square brackets of (\ref{3estimate}) is already nonpositive.  For the second term,
\begin{equation} \label{claimab}
\begin{split}
\lefteqn{\int_{Q_{n-1}} \int_{-c_0}^0 (-Y)(b^i-a^{ij}_j) ((D_Y^N w)^2)_i  \xi^2 dYdX } \\ \le {} &  \int_{Q_{n-1}} \int_{-c_0}^0 (b^n-a^{nj}_j) (D_Y^N w)^2  \xi^2 dYdX   + CI_N.
\end{split}
\end{equation}
Next
\begin{equation} \label{nextclaim}
\begin{split}
\lefteqn{ \int_{Q_{n-1}} \int_{-c_0}^0 2a^{in} (D_Y^N w)(D_Y^N w_i) \xi^2 dYdX } \\ = {} &  \int_{Q_{n-1}} \int_{-c_0}^0 a^{in} ((D_Y^N w)^2)_i \xi^2 dYdX \\ 
  \le {} &   -\int_{Q_{n-1}} \int_{-c_0}^0 D_Ya^{nn} (D_Y^N w)^2 \xi^2 dYdX + CI_N,
\end{split}
\end{equation}
since $a^{in}_i=O(|Y|)$ if $i<n$.

Next, compute 
\begin{equation} \label{long}
\begin{split}
\lefteqn{- 2N\int_{Q_{n-1}} \int_{-c_0}^0 (-Y)(D_Y a^{ij}) (D_Y^N w_j) (D_Y^{N-1} w_i)\xi^2 dYdX }\\
= {} & - 2N\int_{Q_{n-1}} \int_{-c_0}^0 \bigg[ (-Y)\sum_{i,j=1}^{n-1} (D_Y a^{ij}) (D_Y^N w_j) (D_Y^{N-1} w_i)  \\ {} & + (-Y) \sum_{i=1}^{n-1}(D_Y a^{in}) (D_Y^{N+1} w) (D_Y^{N-1}  w_i)  
 + (-Y) \sum_{i=1}^{n-1}(D_Y a^{ni}) (D_Y^N w_i) (D_Y^{N} w) \\ {} &  +  (-Y) (D_Y a^{nn}) (D_Y^{N+1}  w) (D_Y^{N}  w) \bigg] \xi^2 dYdX \\
\le {} & - N\int_{Q_{n-1}} \int_{-c_0}^0 \bigg[  (-Y) \sum_{i,j=1}^{n-1} D_Y \left( (D_Y a^{ij}) (D_Y^{N-1} w_i)(D_Y^{N-1}w_j) \right) 
 \\ 
  {} & + 2\sum_{i=1}^{n-1} (D_Ya^{in}) (D_Y^N w) (D_Y^{N-1}  w_i) 
  + (-Y)(D_Y a^{nn}) D_Y((D_Y^{N}  w)^2)  \bigg]\xi^2 dYdX  \\ {} & + \delta \int_{Q_{n-1}} \int_{-c_0}^0 (D_Y^{N} w)^2 \xi^2 dYdX  +C_{\delta} I_N \\
 \le {}& - N\int_{Q_{n-1}} \int_{-c_0}^0   (D_Y a^{nn}) (D_Y^N w)^2 \xi^2 dYdX 
 + \delta \int_{Q_{n-1}} \int_{-c_0}^0 (D_Y^{N} w)^2 \xi^2 dYdX  + C_{\delta} I_N,
\end{split}
\end{equation}
where to obtain the first inequality, we integrated by parts in $Y$ in the second term of the square brackets.  We also made use of (\ref{uselg}) several times.

But we also have,
\begin{equation} \label{good}
\begin{split}
\lefteqn{ \int_{Q_{n-1}} \int_{-c_0}^0 2N (D_Y^N w) (D_Y a^{in}) (D_Y^{N-1} w_i)\xi^2 dYdX  } \\ = {} & 2N \int_{Q_{n-1}} \int_{-c_0}^0   (D_Y a^{nn})(D_Y^N w)^2  \xi^2 dYdX    \\{} &+ \delta \int_{Q_{n-1}} \int_{-c_0}^0 (D_Y^{N} w)^2 \xi^2 dYdX  + C_{\delta} I_N.
\end{split}
\end{equation}

For $r_0 \le \ell \le N-r_0$,
\begin{equation} \label{mixed}
\begin{split}
\lefteqn{\int_{Q_{n-1}} \int_{-c_0}^0 |Y| |D_Y^N w| \,  |D^{\ell}_Y a^{ij}| \,  |D^{N-\ell} w_{ij} | \xi^2 dYdX } \\
\le {} & C\sum_{i,j} \| a^{ij} \|_{C^\ell} \| w_{ij} \|_{C^{N-\ell}} \left( \int_{Q_{n-1}} \int_{-c_0}^0 |Y| |D_Y^N w|^2 dYdX \right)^{1/2} \\
 \le  {} & C \sqrt{I_N} \sum_{i,j} \| a^{ij} \|_{C^{\ell}} \| w_{ij} \|_{C^{N-\ell}} \\
\le {} & C \sqrt{I_N} \sum_{i,j} \left(  \| a^{ij} \|_{C^{r_0}}  \| w \|_{C^{N-r_0+2}} + \| a^{ij} \|_{C^{N-r_0}}   \| w \|_{C^{r_0+2}} \right)\\
\le {} & CI_N + C C_N \sqrt{I_N} \| w \|_{H^{r_0+n/2+3}} \\
\le {} & C I_N+ C_N^2  I_{r} ,
\end{split}
\end{equation}
for $r\ge r_0 + n/2+4$. Here we used parts (i) and (iii) of Theorem \ref{thmbasic}.  We similarly control terms involving $D^{\ell}_Yb^i$ and $D^{\ell}_Y f$.

Finally, using the Cauchy-Schwarz inequality,
\begin{equation} \label{csi}
\begin{split}
\lefteqn{- \int_{Q_{n-1}} \int_{-c_0}^0 4(-Y)a^{ij} \xi \xi_j (D_Y^Nw) (D_Y^Nw_i) dY dX } \\  {} & \le \int_{Q_{n-1}} \int_{-c_0}^0 (-Y)a^{ij} (D_Y^N w_i) (D^N_Y w_j) \xi^2 dYdX + CI_N.
\end{split}
\end{equation}

Combining (\ref{3estimate}), (\ref{claimab}), (\ref{nextclaim}), (\ref{long}), (\ref{good}), (\ref{mixed}) and (\ref{csi})  gives,
\begin{equation} \label{3estimate2}
\begin{split}
\fbox{2} \le {} &  \int_{Q_{n-1}} \int_{-c_0}^0 \bigg[ -a^{ij} (-Y)(D_Y^N w)_i (D_Y^N w)_j  \\ {} &  +  \left( ( N -1)(D_Y a^{nn}) + b^n - a^{nj}_j \right) (D_Y^N w)^2   
 \bigg]  \xi^2 dYdX
\\ {} & +  4\delta \int_{Q_{n-1}} \int_{-c_0}^0 (D_Y^{N} w)^2 \xi^2 dYdX  + C_{\delta} I_N+C_N^2(I_r+1) \\
 \le {} & - \delta\int_{Q_{n-1}} \int_{-c_0}^0 (D_Y^{N} w)^2 \xi^2 dYdX+  C_{\delta} I_N+C_N^2(I_r+1),
\end{split}
\end{equation}
if we choose $\delta>0$ sufficiently (uniformly) small. The last inequality makes use of our assumptions.
By (A$_2$) then we have $D_Ya^{nn} \le -c<0$ for a small uniform constant $c>0$.  In this case we do not need to use the condition $b^n -a^{nj}_j \le 0$ on the boundary, since we can choose $N$ large.  This gives us a good (negative) term of the order $\int (D_Y^Nw)^2\xi^2$.  

\begin{remark}
If we do not have (A$_2$) but instead assume (B') then we have $D_Ya^{nn} \le 0$ and $b^n - a^{nj}_j \le -c<0$, and the rest of the argument works the same way.   This is what is needed to establish the last assertion of Remark \ref{remarkA1}, (i).
\end{remark}

From now on we have fixed $\delta>0$ so can write $C_{\delta}$ as $C$.

For $\fbox{3}$ we have
\begin{equation} \begin{split}
\fbox{3} \le {}&  \int_{Q_{n-1}} \int_{-c_0}^0 2\bigg[  \sum_{|\bal| = N} (D_{X}^{\bal} w) ( a^{ij} D_{X}^{\bal} w_{ij} + b^i D_{X}^{\bal} w_i) \\ {} &  + \sum_{|\bal| = N-1}\sum_{\ell=1}^{n-1}  N (D_X^{\bal} w_{\ell}) a^{ij}_{\ell} D_{X}^{\bal} w_{ij}) \bigg] \xi^2 dYdX  \\
{} & + \frac{\delta}{3} \int_{Q_{n-1}} \int_{-c_0}^0 (D_Y^{N} w)^2 \xi^2 dYdX
+ CI_N +C_N^2(I_r+1),
\end{split}
\end{equation}
where we argued as in (\ref{mixed}) above to deal with terms involving other derivatives of the coefficients $a^{ij}, b^i, f$.  Hence, integrating by parts,
\begin{equation}\label{2estimate}
\begin{split}
\fbox{3}\le {} &  \int_{Q_{n-1}} \int_{-c_0}^0  \bigg[ - 2 a^{ij} \sum_{|\bal| = N} (D_{X}^{\bal} w_i) (D_{X}^{\bal} w_j) + (b^i - a^{ij}_j)\sum_{|\bal| = N} ((D_{X}^{\bal} w)^2)_i \\
{} & -2 \sum_{|\bal| = N-1}\sum_{\ell=1}^{n-1}  N (D_X^{\bal} w_{\ell j}) a^{ij}_{\ell} D_{X}^{\bal} w_{i} \bigg] \xi^2 dY dX \\
{} & - \int_{Q_{n-1}} \int_{-c_0}^0 \sum_{ |\bal| = N} 4a^{ij} \xi \xi_j (D_{X}^{\bal} w) (D_{X}^{\bal} w_i) dY dX \\ {} & + \frac{2\delta}{3} \int_{Q_{n-1}} \int_{-c_0}^0 (D_Y^{N} w)^2 \xi^2 dYdX+ C I_N + C_N^2 (I_r+1),
\end{split}
\end{equation}
where we used the fact that $a^{in}_{\ell}=0$ on $Y=0$ for $\ell=1, \ldots, n-1$.

We have
\begin{equation}
\begin{split}
\lefteqn{- \int_{Q_{n-1}} \int_{-c_0}^0 2 \sum_{|\bal| = N-1}\sum_{\ell=1}^{n-1}  N a^{ij}_{\ell} (D_X^{\bal} w_{\ell j})  (D_{X}^{\bal} w_{i}) \xi^2 dY dX } \\ 
\le {} & - \int_{Q_{n-1}} \int_{-c_0}^0  \sum_{|\bal| = N-1}\sum_{\ell=1}^{n-1}  N D_{X_{\ell}} \bigg[ a^{ij}_{\ell} (D_X^{\bal} w_{j})  (D_{X}^{\bal} w_{i})  \bigg] \xi^2 dY dX \\
{} & +   \int_{Q_{n-1}} \int_{-c_0}^0  \sum_{|\bal| = N-1}\sum_{\ell=1}^{n-1}  N a^{ij}_{\ell \ell} (D_X^{\bal} w_{j})  (D_{X}^{\bal} w_{i})   \xi^2 dY dX \\
\le {} & \frac{\delta}{3} \int_{Q_{n-1}} \int_{-c_0}^0 (D_Y^{N} w)^2 \xi^2 dYdX + C I_N,
\end{split}
\end{equation}
where for the last line we integrate by parts in $X_{\ell}$ in the first term.

Using the Cauchy-Schwarz inequality, we estimate
\begin{equation}
\begin{split}
\lefteqn{
 -\int_{Q_{n-1}} \int_{-c_0}^0 \sum_{|\bal| = N} 4a^{ij} \xi \xi_j (D_{X}^{\bal} w) (D_{X}^{\bal} w_i) dY dX } \\
\le {} & \int_{Q_{n-1}} \int_{-c_0}^0 a^{ij} \sum_{|\bal| = N} (D_{X}^{\bal} w_i) (D_{X}^{\bal} w_j) \xi^2 dYdX +CI_N.
\end{split}
\end{equation}

We use the boundary condition (B) on $Y=0$ to get
\begin{equation} \label{2estimate2}
\begin{split}
\fbox{3} \le {} &  - \int_{Q_{n-1}} \int_{-c_0}^0 \sum_{|\bal|=N} a^{ij} (D_{X}^{\bal} w)_i (D_{X}^{\bal} w)_j \xi^2  dYdX  \\
{} & +   \delta \int_{Q_{n-1}} \int_{-c_0}^0 (D_Y^{N} w)^2 \xi^2 dYdX + CI_N + C_N^2 ( I_r+1).
\end{split}
\end{equation}

Finally, we bound
\begin{equation}\label{middle}
\begin{split}
\lefteqn{\frac{d}{dt} \int_{U_0} \sum_{|\mathbf{a}|=N} (D^{\mathbf{a}}w)^2 \xi_0^2 dx } \\ \le {} &  -\frac{\ve}{2} \int_{U_0} \sum_{|\mathbf{a}|=2N} (D^{\mathbf{a}}w)^2 \xi_0^2 dx + C\ve \| w\|_{H^{2N-1}(\Omega_{c_0/4})} \\
{} & - \int_{U_0} \sum_{|\mathbf{a}|=N} a^{ij} (D^{\mathbf{a}} w_i)(D^{\mathbf{a}} w_j)\xi_0^2dx + CI_N + C_N^2 (I_r+1), 
\end{split}
\end{equation}
after integration by parts.  In particular, lower order terms where derivatives land on cut-off functions can be controlled by the first  four terms on the right hand side of (\ref{middle}). Note that $(a^{ij})$ has a positive lower bound on $U_0$.

Combining (\ref{middle}) with (\ref{equationbox1}), (\ref{3estimate2}) and (\ref{2estimate2}) and summing $\lambda$ from $1$ to $K$, gives (\ref{kle}).  Here we are making use of terms on the right hand side of (\ref{equationbox1}), together with Lemma \ref{lemmagarding}, to control $C\ve \| w\|_{H^{2N-1}(\Omega_{c_0/4})}$.

This completes the proof of Lemma \ref{klemma} for $n\ge 2$. For dimension $n=1$ the proof is simpler since there are no $X$ coordinates.  The boundary condition (B) is used only to establish (\ref{2estimate2}) and hence is not needed in this case.
\end{proof}

We now complete the proof of Theorem \ref{theoremmain}.

\begin{proof}[Proof of Theorem \ref{theoremmain}]  First we will show the existence of a unique smooth solution of (\ref{le}), and then we will prove the estimate (\ref{tameest}).  In what follows, denote by $\tilde{C}$ a constant that may depend on bounds for the right hand side of (\ref{tameest}), and in particular may be uncontrolled as $N\rightarrow \infty$.

For each fixed large even integer $N$, denote by $w_{\ve}(t)$ the solution of (\ref{ler}).  From Lemma \ref{klemma},  we have $I_N(w_{\ve}(t)) \le \tilde{C}$ and hence in particular
\begin{equation} \label{sbbd}
\int_0^T \| w_{\ve}(t) \|^2_{H^{N-1}} dt \le \tilde{C}.
\end{equation}
It follows that for a sequence $\ve_{\ell} \rightarrow 0$, we have $w_{\ve_{\ell}} \rightharpoonup w$ weakly in $L^2(\Omega_0 \times [0,T])$.  Moreover,  for each multi-index $\mathbf{a}$ with $|\mathbf{a}|\le N-1$, 
 the 
weak derivative $D^{\mathbf{a}}w$ exists and is in $L^2 (\Omega_0 \times [0,T])$.  After passing to a subsequence, we may assume that $D^{\mathbf{a}} w_{\ve} \rightharpoonup D^{\mathbf{a}}w$ in $L^2(\Omega_0 \times [0,T])$ as $\ve=\ve_{\ell} \rightarrow 0$.

Now let $\varphi$ be a smooth test function compactly supported in $\Omega_0 \times (0,T)$.  We have
$$\int_0^T \int_{\Omega_0} w_{\ve}  D^{\mathbf{a}} D_t \varphi dx dt = \int_0^T \int_{\Omega_0} \left( \ve L_{2N} w_{\ve} - a^{ij}(w_{\ve})_{ij} - b^i (w_{\ve})_i - fw_{\ve}-g) \right)D^{\mathbf{a}} \varphi dx dt.$$
We now let $\ve = \ve_{\ell} \rightarrow 0$.  Using  integration by parts we have $$\int_0^T \int_{\Omega_0} \ve (L_{2N}w_{\ve}) (D^{\mathbf{a}} \varphi) dxdt=O(\ve),$$ and hence for $|\mathbf{a}|\le N-3$,
$$\int_0^T \int_{\Omega_0} w  D^{\mathbf{a}} D_t \varphi dx dt = (-1)^{1+|\mathbf{a}|} \int_0^T \int_{\Omega_0} D^{\mathbf{a}} (a^{ij} w_{ij} + b^i w_i + fw +g) \varphi dxdt.$$
Then the weak derivative $D^{\mathbf{a}} D_tw$ exists, is bounded in $L^2(\Omega_0 \times [0,T])$, and is equal to
$$D^{\mathbf{a}} D_t w = D^{\mathbf{a}} ( a^{ij} w_{ij} + b^i w_i + fw +g).$$
In particular we have
\begin{equation} \label{mew}
\frac{\partial w}{\partial t} = a^{ij} w_{ij} + b^i w_i + fw +g, \quad \textrm{on } \Omega_0 \times (0,T).
\end{equation}
Replacing $D_t \varphi$ by $D_t^m \varphi$ for $m = 2, 3,  \ldots$ and repeating the argument above we see that for $2m +|\mathbf{a}| \le N-1$, the weak derivative $D_t^m D^{\mathbf{a}} w$ exists and 
$$\| D_t^m D^{\mathbf{a}} w \|_{L^2(\Omega_0\times [0,T])} \le \tilde{C}.$$
For $N$ sufficiently large, depending only on $n$, we see that $w$ is $C^2$ in space and $C^1$ in time in $\ov{\Omega}_0\times [0,T]$, and hence (\ref{mew}) holds in the classical sense on $\ov{\Omega}_0 \times [0,T]$.  By Proposition \ref{propunique}, $w$ is the unique such solution to (\ref{mew}).

Since $N$ was arbitrary, we may let $N$ tend to infinity to obtain a unique smooth solution $w$ to (\ref{mew}).

Next, we prove the estimate (\ref{tameest}).  From Lemma \ref{klemma}, after taking $\ve \rightarrow 0$, we have
$$\frac{d}{dt} I_{N} e^{-Ct} \le C_N^2 (I_r+1)e^{-Ct},$$
where $I_{N}$ and $I_r$ are defined in terms of the solution $w$ of (\ref{mew}).  Since $r$ is fixed, we have $I_r\le C$ and hence integrating in time we obtain
$$I_{N} \le CC_N^2.$$
This implies that for any $\ell$,
\[
\begin{split}
\| w \|_{C^{\ell}} \le {} & \| w \|_{H^{\ell+\lceil n/2 \rceil+1}}  \\
\le {} & C(\ell)  \sup_{t \in [0,T]} \bigg( 1+ \sum_{i,j} \|a^{ij}\|_{C^{\ell+r}} + \sum_i \|b^i\|_{C^{\ell+r}}+ \| f\|_{C^{\ell+r}} + \|g\|_{C^{\ell+r}}\bigg),
\end{split}
\]
 increasing $r=r(n)$ if necessary.    This completes the proof.
\end{proof}

\section{Transformation to an equation on a fixed domain} \label{sectiontrans}

In this section we transform the porous medium equation (\ref{PME2}) to an equation on the fixed domain $\Omega_0$.  This  is also carried out by Daskalopoulos-Hamilton \cite{DH98}, but in a less explicit way.  A key point, which we will see below, is that the boundary data on $\partial \Omega_0$ is undetermined for $t>0$.

We work under the assumptions of Theorem \ref{mainthm}.  In particular, the nondegeneracy condition on $v_0$ means that 
\begin{equation} \label{nondegen}
v_0 + (\nabla v_0)^2 \ge c>0, \quad \textrm{on } \ov{\Omega}_0,
\end{equation}
for a uniform constant $c>0$.

For each  fixed $t \ge 0,$ the graph $y=v(x,t)$ is a set of points $(x,y)=(x_1, \ldots, x_n, y)$ in  $\mathbb{R}^{n} \times [0,\infty)$.   Define a vector-valued map $V: \Omega_0  \rightarrow \mathbb{R}^{n+1}$ by
$$V(x):= (-\nabla v_0(x), v_0(x)) = (-\partial_1 v_0(x), \ldots, -\partial_n v_0(x), v_0(x)).$$

For $x\in \Omega_0$, the vector $V(x)$ is transverse to the graph $y=v_0(x)$ in $\mathbb{R}^{n+1}$ at the point $(x, v_0(x))$.  Note that $V(x)$ is parallel to the $y=0$ hyperplane for $x \in \partial \Omega_0$.  Moreover, by (\ref{nondegen}), we have $|V| \ge c'>0$.  

Suppose that we have a solution $v(x,t)$ of the porous medium equation for $t \in [0,T]$ which is smooth on $\bigcup_{t \in [0,T]} \ov{\Omega}_t$.  For a sufficiently small $\ve>0$, and all $x \in \ov{\Omega}_0$, 
the parametrized line segment $$s \mapsto (x, v_0(x)) + sV(x), \quad -\ve < s <\ve,$$ intersects the graph $y=v(x,t)$ at a unique value of $s$, where we are shrinking $T>0$ if necessary.  We write $h(x,t)$ for this value of $s$.  
The function $h(x,t)$ satisfies
\begin{equation} \label{h}
(1+  h(x,t) )v_0(x) = v(x-\nabla v_0(x)h(x,t), t),
\end{equation}
and indeed (\ref{h})  can be used to define $h$.  Observe that $h(x,0)=0$.

Conversely, a smooth $h(x,t)$ for $t\in [0,T]$, which is sufficiently small, determines the graphs $y=v(x,t)$.

We now wish to write down the porous medium equation $v_t = (m-1)v \Delta v + |\nabla v|^2$  in terms of $h(x,t)$.  Differentiating (\ref{h}) with respect to $x_i$ we obtain
$$(1+h)v_{0,i} + h_i v_0 = v_j \left( \delta_{ij} - h_i v_{0,j} - h v_{0,ij} \right),$$
where we are evaluating $v_i$ and $v_j$ at $(x-\nabla v_0(x)h(x,t),t)$ and summing over the repeated $j$ index.  Hence
\begin{equation} \label{vxh}
v_j = A^{ji}(  (1+h)v_{0,i} +h_i v_0),
\end{equation}
where $A^{ji}$ is defined by 
\begin{equation} \label{defnA}
A^{ji}(\delta_{ik} - h_i v_{0,k} - h v_{0,ik} ) = \delta_{jk}
\end{equation} (we assume that $h$ and $|\nabla h|$ are small).
Note that $(A^{ji})$ is in general not a symmetric matrix, because $h_i v_{0,k}$ may not be symmetric in $i,k$.  Writing $A=(A^{ji})$ and $C= C_{ik} = \delta_{ik} - h_i v_{0,k} - hv_{0,ik}$ we have $AC=\textrm{Id} = CA$ and $A^T C^T =\textrm{Id}$, so in particular
$$A^{ij}(\delta_{ik} - h_k v_{0,i} - h v_{0,ik} ) = \delta_{jk}.$$
Observe that
$$\partial_{k} A^{ji} = A^{\ell i} A^{jr} (h_{rk} v_{0,\ell} + h_r v_{0,\ell k} + h_k v_{0,r\ell}+ hv_{0,r\ell k}).$$

Applying $\partial_k$ to (\ref{vxh}), remembering that we are evaluating  $v_j$ at $(x-\nabla v_0(x)h(x,t),t)$,
\begin{equation} \label{vxxh}
\begin{split}
\lefteqn{v_{ji} \left( \delta_{ik} - h_k v_{0,i} - h v_{0,ik} \right) } \\ = {} & A^{ji}((1+h) v_{0,ik} + h_k v_{0,i} + h_i v_{0,k} + h_{ik} v_0) \\ {} & + A^{\ell i} A^{jr} (h_{rk} v_{0,\ell} + h_r v_{0,\ell k} + h_k v_{0,r\ell}+ hv_{0,r\ell k})  (  (1+h)v_{0,i} +h_i v_0).
\end{split}
\end{equation}
Multiplying by $A^{qk}$ and summing over $k$ we get
\begin{equation} \label{vxxh2}
\begin{split}
v_{jq} = {} & A^{qk}A^{ji}((1+h) v_{0,ik} + h_k v_{0,i} + h_i v_{0,k} + h_{ik} v_0) \\ {} & + A^{qk}A^{\ell i} A^{jr} (h_{rk} v_{0,\ell} + h_r v_{0,\ell k} + h_k v_{0,r\ell}+ hv_{0,r\ell k})  (  (1+h)v_{0,i} +h_i v_0).
\end{split}
\end{equation}

On the other hand, differentiating  (\ref{h}) with respect to $t$ gives
$$h_t v_0 = -v_i v_{0,i}h_t+v_t$$
and from $v_t=(m-1)v\Delta v+ |\nabla v|^2$ we obtain
$$h_t (v_0+v_i v_{0,i}) = (m-1)(1+h)v_0 \Delta v +|\nabla v|^2.$$
Substituting from (\ref{vxh}), (\ref{vxxh2}), we see that $h(x,t)$ solves the following equation:
\begin{equation} \label{Fh}
\begin{split}
h_t  =  {} & \frac{1}{v_0 + A^{ij} ((1+h)v_{0,j}+h_j v_0) v_{0,i}} \cdot \\ {} & \bigg[ (m-1)(1+h)v_0 \bigg( A^{jk}A^{ji}((1+h) v_{0,ik}   + h_k v_{0,i} + h_i v_{0,k} + h_{ik} v_0) \\ {} & + A^{jk}A^{\ell i} A^{jr} (h_{rk} v_{0,\ell} + h_m v_{0,\ell k} + h_k v_{0,r\ell}+ hv_{0,r\ell k})  (  (1+h)v_{0,i} +h_i v_0) \bigg) \\
+ {} & A^{ji}A^{jk} ((1+h)v_{0,i} + h_i v_0)((1+h) v_{0,k} + h_k v_0) \bigg].
\end{split}
\end{equation}
Note that $(x,t) \in \ov{\Omega}_0 \times [0,T]$, so the domain is fixed.  Importantly, there are no boundary conditions for $h$.  The initial data is $h(x,0)=0$.

Write  equation (\ref{Fh}) for $h$ as $h_t = F(h, x)$, where $F$ is a second order differential operator in $h$.  We can write $F$ as follows.
\begin{equation} \label{F2h}
\begin{split}
F(h,x) = {} & \frac{1}{v_0 + A^{ij} ((1+h)v_{0,j}+h_j v_0) v_{0,i}} \bigg[  (m-1)(1+h)^2 v_0 A^{jk}A^{\ell i} A^{jr} v_{0,\ell}v_{0,i}  h_{r k}  \\ {} & + v_0^2 \varphi_{ij}(h,\nabla h, x) h_{ij} 
 +  (1+h)^2 A^{ji} A^{jk} v_{0,i}v_{0,k} + v_0 \psi (h, \nabla h, x)\bigg],
\end{split}
\end{equation}
for smooth functions $\varphi_{ij}$, $\psi$.  Define $L$ as the linearization of $F$ at small $h$, namely,
$$L(h,w,x) = \frac{d}{d\tau} F(h+\tau w, x)|_{\tau=0},$$ 
which takes the form
\begin{equation} \label{Ldefn}
L(h,w,x) = a^{ij} w_{ij} + b^i w_i + f w,
\end{equation}
for smooth functions $a^{ij}, b^i, f$ on $\ov{\Omega}_0$, which depend on $h=h(x,t)$ and its derivatives and $x$.  The symmetric matrix $(a^{ij}(x,t))$ is positive definite for $x \in \Omega_0$ and vanishes on the boundary $\partial \Omega_0$.

We have the following proposition.

\begin{proposition} \label{propPME}
Write $\nu= (\nu_1, \ldots, \nu_n)$ for the outward pointing normal to $\Omega_0$.  On $\partial \Omega_0$ we have
\begin{equation} \label{propeq1}
(a^{ij})=0, \quad (\partial_k a^{ij}) \nu_k \nu_{i} \nu_j <0, 
\end{equation}
and
\begin{equation} \label{propeq2}
\begin{split}
(b^i - \partial_j a^{ji}) \nu_i <  {} & 0, \quad \textrm{if } 1<m<2 \\
(b^i - \partial_j a^{ji}) \nu_i = {} & 0, \quad \textrm{if } m=2 \\
(b^i - \partial_j a^{ji}) \nu_i > {} & 0, \quad \textrm{if } m>2. \\
\end{split} 
\end{equation}
and 
\begin{equation} \label{propeq3}
b^i \nu_i \le 0.
\end{equation}
In particular, conditions \emph{(A$_1$)}, \emph{(A$_2$)} hold, and Fichera condition \emph{(B)} holds if $1 < m \le 2$.
\end{proposition}
\begin{proof}
It follows immediately from (\ref{Fh}) that the components $a^{ij}$ all vanish on the boundary since $v_0$ does.

It is convenient to assume that $p$ is the origin and $\nu = (0, \ldots, 0, 1)$.  Hence the tangent plane to $\partial \Omega_0$ is the plane $x_n=0$ and we have $v_{0,i}=0$ for $i=1, \ldots, n-1$ and $v_{0,n}<0$. From (\ref{F2h}) we compute at the origin,
 \[
 \begin{split}
 \partial_n a^{nn} = {} & \frac{(m-1)(1+h)^2 (v_{0,n})^3 A^{nn} A^{jn}A^{jn}}{v_0 + A^{ij} ((1+h)v_{0,j}+h_j v_0) v_{0,i}}  \\
 = {} & (m-1)(1+h) v_{0,n} A^{jn}A^{jn}<0,
 \end{split}
 \]
 and in particular this completes the proof of (\ref{propeq1}).

The quantity $(b^i - \partial_j a^{ji}) \nu_i$ becomes
\begin{equation}\label{KN2}
b^n - \partial_j a^{jn}.
\end{equation}
Observe that because of the $v_0$ term in (\ref{F2h}) which vanishes at the origin and $v_{0,i}(0)=0$ for $i=1, \ldots, n-1$, we see that $\partial_j a^{k\ell}=0$ for $j=1, \ldots, n-1$.  Hence we only have to compute, at the origin,
$$b^n - \partial_n a^{nn}.$$

 Next, noting that at the origin, we compute the variation with respect to $h$ as follows,
 $$\delta A^{jn} 
 = A^{ji} A^{nn} ( \delta h)_i v_{0,n} + A^{ji} A^{\ell n} (\delta h) v_{0,i\ell},$$
 and in particular
 $$\delta A^{nn} =  A^{ni} A^{nn}  (\delta h)_i v_{0,n} + A^{ni}A^{jn} (\delta h)  v_{0,ij}.$$

 Then at the origin,
 \begin{equation} \label{bformula}
 \begin{split}
 b^n =  {} & \frac{-1}{(v_0 + A^{ij} ((1+h)v_{0,j}+h_j v_0) v_{0,i})^2}(1+h) (v_{0,n})^5 A^{nn}A^{nn} (1+h)^2 A^{jn}A^{jn}  \\
 {} & + \frac{2}{v_0 + A^{ij}((1+h)v_{0,j} + h_j v_0) v_{0,i}}  (1+h)^2 A^{jn} A^{nn} (v_{0,n})^3 A^{jn} \\ 
 = {} & (1+h)v_{0,n} A^{jn} A^{jn}.
 \end{split}
 \end{equation}
From
$$b^n -\partial_n a^{nn}  = (2-m)(1+h)v_{0,n} A^{jn} A^{jn},$$
we obtain (\ref{propeq2}).

Assertion (\ref{propeq3}) follows from (\ref{bformula}).
 \end{proof}

\section{Proof of Theorem \ref{mainthm}} \label{sectionlast}

In this section we apply the Nash-Moser Inverse Function Theorem to give the proof of Theorem \ref{mainthm}.
 We use the results above to prove the existence of a smooth solution $h : \ov{\Omega}_0 \times [0,T] \rightarrow \mathbb{R}$ to
 \begin{equation} \label{pmeh}
 \begin{split} 
 h_t(x,t) = {} & F(h,x), \quad (x,t) \in \ov{\Omega}_0 \times [0,T],\\
 h(x,0) = {} & 0, \quad x \in \Omega_0,
 \end{split}
 \end{equation}
where $F$ is the second order operator in $h$ defined by (\ref{F2h}). We will shrink $T>0$ if necessary.  This will then immediately imply Theorem \ref{mainthm}.
 
We will use $\| \cdot \|_{\ell}$ to denote the norms
$$\| f\|_{\ell}:= \max_{2k+|\mathbf{a}|\leq \ell} \ \max_{\ov{\Omega}_0\times[0, T]}| D_x^\mathbf{a} D_t^k f|.$$

  First we construct a ``formal solution'' $\tilde{h}$ of $\Phi(\tilde{h}):= \tilde{h}_t - F(\tilde{h},x)$ at $t=0$, by which we mean a smooth function $\tilde{h}$ on $\ov{\Omega}_0 \times [0,T]$ with the property that
\begin{equation} \label{formasolution}
D^j_t \Phi(\tilde{h}) (x, 0)=0,\quad \text{for } x\in \ov{\Omega}_0, \ j=0, 1,2, \ldots
\end{equation}
Note that if $h$ solves $\Phi(h)=0$ on $\ov{\Omega}_0 \times [0,T]$, then writing this as $h_t=F(x, h, D h, D^2 h)$, differentiating $k$ times in $t$ then evaluating at $t=0$ gives $D^k_t h (x, 0)$ in terms of the partial derivatives $D^a_{2} D^b_{3} D^c_{4} F(x, 0,0,0)$ for $1\leq a+b+c \leq k-1$ and $D^l_t h (x, 0)$ for $0\leq l \leq k-1$.  
From a recursive argument starting from $a_0(x) := h(x,0)=0$ we obtain formulae for $a_j(x):=D_t^j h(x,0)$ involving only the partial derivatives $D^a_{2} D^b_{3} D^c_{4} F(x,0,0,0)$ for $1\leq a+b+c \leq k-1$.  Now let $\tilde{h}$ be a smooth function on $\ov{\Omega}_0 \times [0,T]$ whose Taylor series at $t=0$ is 
$$\sum_{j=0}^{\infty} \frac{a_j(x)}{j!} t^j.$$
Note that since $\tilde{h}(x,t) \rightarrow 0$ uniformly in $x$ as $t\to 0$, we may and will assume that $T$ is chosen suitably small so that $\| \tilde{h} \|_{\ell}$ are as small as we like. By construction, $\tilde{h}$ solves (\ref{formasolution}).

Define a function space
$$\mathcal{B}:=\{h(x, t)\in C^{\infty}(\ov{\Omega}_0 \times[0, T]) \ | \  	D_t^j h(x, 0)=0, \ \textrm{for } j=0,1,2,\ldots \},$$
and equip it with the seminorms $\| \cdot \|_{\ell}$ defined above.

 \begin{lemma}
The seminorms $\|\cdot \|_{\ell}$ give $\mathcal{B}$ the structure of a tame space as in \cite{H}. 
 \end{lemma}
 \begin{proof} 
This follows from  the proof of \cite[Theorem II.1.3.6]{H}. We use coordinates $(\xi_1, \ldots, \xi_n, \tau)$ on $\mathbb{R}^{n+1}$ and take the weight function
to be $$w = \log (1+|\xi|) +\frac{1}{2} \log (1+|\tau|)$$ for $|\xi|^2 = \xi_1^2 + \cdots \xi_n^2$.
 \end{proof}

Recall that Frechet spaces are metrizable.  Indeed, given seminorms $\| \cdot \|_k$ one can write down a metric,
$$d(f, g) = \sum_{k=0}^{\infty} 2^{-k} \frac{ \| f-g\|_k}{1+\| f-g\|_k}.$$
Let $\mathcal{D}$ be a ball of small radius $\sigma>0$, with respect to this metric, centered at the origin in $\mathcal{B}$.  
Define an operator $\hat{\Phi}: \mathcal{D} \subset \mathcal{B}  \rightarrow \mathcal{B}$ by
\begin{equation}
\hat{\Phi}(h) = (h+\tilde{h})_t - F(h+\tilde{h},x),
\end{equation}
where $\tilde{h}$ is our formal solution given above. Our goal is to find $h$ so that $\hat{\Phi}(h)=0$, so that our desired solution is $h+\tilde{h}$.

Observe that $F(h,x)$ is not defined for all $h \in \mathcal{B}$ because of the denominators which may be zero for some $h$.  However, 
shrinking $\sigma>0$ and $T>0$ as necessary, we may assume that for a fixed $r$, the $C^r$ norms of $h+\tilde{h}$ for $h \in \mathcal{D}$ are as small as we like.  Moreover, $D_t^j\hat{\Phi}(h)=0$ at $t=0$ for all $j$ by the definition of $\tilde{h}$ and $\mathcal{B}$. 
Hence the operator is well-defined.
Since $\hat{\Phi}$ is a differential operator of degree $2$ it satisfies the ``tame estimate''
$$\| \hat{\Phi} (h)\|_{\ell} < C(\ell)(1+ \| h\|_{\ell+2}).$$
The operator $\hat{\Phi}$ is a \emph{smooth tame map} as defined in \cite[p. 143]{H}.

In order to apply the Nash-Moser Inverse Function Theorem, we need to consider the linearization $D\hat{\Phi}$ of $\hat{\Phi}$, which is a map
$$D\hat{\Phi} : \mathcal{D} \times \mathcal{B} \rightarrow \mathcal{B}.$$
We need to show that for all $h \in \mathcal{D}$ the equation
\begin{equation} \label{DPhih}
D\hat{\Phi} (h) w = g,
\end{equation}
has a unique solution $w = (D\hat{\Phi})^{-1}(h)g$ for all $g \in \mathcal{B}$ and that the family of inverses
$$(D\hat{\Phi})^{-1} : \mathcal{D} \times \mathcal{B} \rightarrow \mathcal{B}$$
is a smooth tame map. 

The equation (\ref{DPhih}) is given by 
\begin{equation} \label{DPhih2}
w_t = Lw + g
\end{equation}
where $L$ is the operator $Lw = a^{ij}w_{ij} + b^i w_i + fw$ given in (\ref{Ldefn}), with coefficients $a^{ij}$, $b^i$, $f$ as described in Section \ref{sectiontrans} except that $h$ in the coefficients is replaced by $h+\tilde{h}$.  We assume now that either $n=1$ or $1<m\le 2$, so that by Proposition \ref{propPME} the assumptions of Theorem \ref{theoremmain} are satisfied.  Then  for any $g \in \mathcal{B}$, the equation (\ref{DPhih2}) has a unique solution $w \in \mathcal{B}$.  Moreover, for $\ell =0,1,2,\ldots,$ we have
	\begin{equation}\label{tameest2}
		\|w\|_{C^{\ell}} \leq C(\ell) \sup_{t\in[0,T]}\bigg( 1+\sum_{i,j} \|a^{ij}\|_{C^{\ell+r}} + \sum_i \|b^i\|_{C^{\ell+r}}+ \| f\|_{C^{\ell+r}} + \|g\|_{C^{\ell+r}}\bigg),
	\end{equation}
	for $r=r(n)$ and 
 constants $C(\ell)$ depending only on $n, \ell, T, \Omega_0$ and $C^r$ bounds of $a^{ij}, b^i, f$.
 
 On the other hand, the coefficients $a^{ij}, b^i, f$ are smooth functions of $h$ and $x$, and by the equation satisfied by $h$, we can replace a time derivative of $h$ by two spatial derivatives. Hence, increasing $r$ if necessary, we have
 \begin{equation} \label{tamingw}
 \| w\|_{\ell} \le C(\ell) (1+ \| h \|_{\ell+r} + \| g \|_{\ell+r} ),
 \end{equation}
 where we are using the equation satisfied by $w$ in order to apply (\ref{tameest2}). 
 Observe that $(D\hat{\Phi})^{-1} : \mathcal{D} \times \mathcal{B} \rightarrow \mathcal{B}$ is continuous.  Indeed, suppose that $(h_k, g_k)$ are in the domain of this operator with $(h_k, g_k) \rightarrow (h,g)$ and $w_k$ satisfies $D\Phi(h_k) w_k = g_k$.  Then the estimate (\ref{tamingw}) and Proposition \ref{propunique} imply that 
 $w_k$ converges to the unique $w \in \mathcal{B}$ solving $D\hat{\Phi}(h)w = g$. 
From \cite{H}, the estimate (\ref{tamingw}) implies that  the family of inverses $(D\hat{\Phi})^{-1} : \mathcal{D} \times \mathcal{B} \rightarrow \mathcal{B}$ is a smooth tame map.  

We can now apply the Nash-Moser Inverse Function Theorem \cite[III.1.1.1]{H}, which states that $\hat{\Phi}$ is locally invertible and each local inverse $\hat{\Phi}^{-1}$ is a smooth tame map.  

Define $\tilde{\rho}:= \hat{\Phi}(0) = \tilde{h}_t - F(\tilde{h},x) \in \mathcal{B}$.  Then there exists a neighborhood $\mathcal{U}$ of $\tilde{\rho}$ such that for each $\rho \in \mathcal{U}$ there exists $h\in \mathcal{D}$ with $\hat{\Phi}(h)=\rho$, namely
$$(h+\tilde{h})_t = F(h+\tilde{h},x) + \rho.$$

It is now standard to complete the proof (see for example \cite[p. 195-196]{H}).  
For a small $\ve>0$ define a function $\rho_{\ve}$ by
$$\rho_{\ve}(x,t) = \begin{cases}
	0 & t\in [0, \varepsilon) \\
	(\tilde{h}_t - F(\tilde{h},x))(x, t-\ve) & t\in [\varepsilon, T]  \\
\end{cases} $$
 By our construction of $\tilde{h}$, the function $\rho_{\ve}$ is smooth on $\ov{\Omega}_0 \times [0,T]$. In addition, we have $\| \rho_{\ve} -\tilde{\rho} \|_{\ell} \rightarrow 0$ as $\ve \rightarrow 0$ for every $\ell$.  Hence for $\ve>0$ sufficiently small we have $\rho_{\ve} \in \mathcal{U}$, where $\mathcal{U}$ is the neighborhood of $\tilde{\rho}$ given by the Nash-Moser Theorem. Hence there exists a solution $h$ to 
 $$(h+\tilde{h})_t = F(h+\tilde{h}, x)+ \rho_{\ve} \quad \textrm{on } \ov{\Omega}_0 \times [0,T].$$
 In particular, we have
 $$(h+\tilde{h})_t = F(h+\tilde{h},x), \quad \textrm{on } \ov{\Omega}_0 \times [0,\ve].$$
 This completes the proof of the Theorem \ref{mainthm}.

\end{document}